\documentclass[12pt]{amsart}
\usepackage{amsmath,amsfonts,amssymb,amsthm}
\usepackage{graphicx,color}
\usepackage{epsfig, tikz}
\DeclareMathRadical{\sqrtsign}{symbols}{"70}{largesymbols}{"70}
\newcommand{\bb}{\mathbb}

\newcommand{\integers}{{\bb Z}}
\newcommand{\natls}{{\bb N}}
\newcommand{\ratls}{{\bb Q}}
\newcommand{\reals}{{\bb R}}
\newcommand{\proj}{{\bb P}}
        
\newlength{\figboxwidth}             
\renewcommand{\bold}[1]{\medskip \noindent {\bf #1 }\nopagebreak}

\newcommand{\cross}{\times}

\newcommand{\st}{\;\: : \;\:}         

\newcommand{\zed}{\integers}

\def\@ifundefined#1#2#3%
  {\expandafter\ifx\csname#1\endcsname\relax#2\else#3\fi}

\@ifundefined{theoremstyle}{
}{
\theoremstyle{plain} 
}
\newtheorem{theorem}{Theorem}[section]
\newtheorem{prop}[theorem]{Proposition}
\newtheorem{proposition}[theorem]{Proposition}
\newtheorem{lemma}[theorem]{Lemma}

\newtheorem{corollary}[theorem]{Corollary}

\@ifundefined{theoremstyle}{
}{
\theoremstyle{definition} 
}
\newtheorem{definition}[theorem]{Definition}

\theoremstyle{remark}
\newtheorem{remark}[theorem]{Remark}
\newtheorem*{remark*}{Remark}

\newcommand{\cH}{{\mathcal H}}

\newcommand{\cM}{{\mathcal M}}
\newcommand{\cN}{{\mathcal N}}

\newcommand{\cU}{{\mathcal U}}
\newcommand{\cV}{{\mathcal V}}

\mathchardef\GG="321D
\newcommand{\mc}[1]{{}}  
\newcommand{\mcc}[1]{{}}

\numberwithin{equation}{section}
\newcommand{\noz}{n}

\theoremstyle{plain}
\newtheorem{thm}{Theorem}

\newtheorem{defin}{Definition}

\begin{document}
\title[Birkhoff and Oseledets genericity]{Every flat surface is Birkhoff and Oseledets generic in almost every direction}
\author[J. Chaika]{Jon Chaika}
\author[A. Eskin]{Alex Eskin}
\maketitle
\begin{abstract} We prove that the Birkhoff Pointwise Ergodic Theorem and the Oseledets Multiplicative Ergodic Theorem hold for every flat surface in almost every direction. The proofs rely on the strong law of large numbers, and on recent rigidity results for the action of the upper triangular subgroup of SL(2,R) on the moduli space of flat surfaces. Most of the results also use a theorem about continuity of splittings of the Kontsevich-Zorich cocycle recently proved by S. Filip.
\end{abstract}
\section{Introduction}
\label{sec:intro}

\bold{Flat surfaces and strata.}
Suppose $g \ge 1$, and 
let $\alpha = (\alpha_1,\dots, \alpha_\noz)$ be a partition of $2g-2$,
and 
 $\cH(\alpha)$ be a stratum of Abelian differentials,
i.e. the space of pairs $(M,\omega)$ where $M$ is a Riemann surface
and $\omega$ is a holomorphic $1$-form on $M$ whose zeroes have
multiplicities $\alpha_1 \dots \alpha_\noz$. The form $\omega$ defines a
canonical flat metric on $M$ with conical singularities at the zeros
of $\omega$. Thus we refer to points of $\cH(\alpha)$ as
{\em flat surfaces} or {\em translation surfaces}. For an introduction
to this subject, see the survey \cite{Zorich:survey}. 

\bold{Affine measures and manifolds.} Let
$\cH_1(\alpha) \subset \cH(\alpha)$ denote the subset of surfaces of
(flat) area $1$. 
An affine invariant manifold is
a closed subset of $\cH_1(\alpha)$ which is invariant under the
$SL(2,\reals)$ action and which in \textit{period coordinates} (see
\cite[Chapter ~3]{Zorich:survey}) looks like an affine subspace. Each affine invariant
manifold $\cM$ is the support of an ergodic $SL(2,\reals)$ invariant
probability measure $\nu_\cM$. Locally, in period coordinates, this
measure is (up 
to normalization) the restriction of Lebesgue measure to the subspace
$\cM$,  see \cite{EM} for the precise definitions. 
It is proved in \cite{EMM} that the closure of any
$SL(2,\reals)$ orbit is an affine invariant manifold.

The most importatant case of an affine invariant manifold is a
connected component a stratum $\cH_1(\alpha)$. In this case, the
associated affine measure is called the Masur-Veech or Lebesgue
measure \cite{masur:interval}, \cite{veech:gauss}.

\bold{The Teichm\"uller geodesic flow.} Let 
\begin{displaymath}
g_t = \begin{pmatrix} e^{t} & 0 \\ 0 & e^{-t} \end{pmatrix} \qquad
r_\theta = \begin{pmatrix} \cos \theta & \sin \theta \\ - \sin \theta
  & \cos \theta \end{pmatrix}. 
\end{displaymath}
The element $r_\theta \in SL(2,\reals)$ acts by $(M,\omega) \to
(M,e^{i\theta} \omega)$. This has the effect of rotating the flat
surface by the angle $\theta$. 
The action of $g_t$ is called the \textit{Teichm\"uller geodesic
  flow}.   The orbits of $SL(2,\reals)$ are called
\textit{Teichm\"uller disks}. 

\bold{A variant of the Birkhoff ergodic theorem.}
We use the notation $C_c(X)$ to denote the space of continuous
compactly supported functions on a space $X$. 

One of our main results is the following:
\begin{theorem}
\label{theorem:birkhoff:flow}
Suppose $x \in \cH_1(\alpha)$. Let $\cM = \overline{SL(2,\reals) x}$
be the smallest affine invariant manifold containing $x$. Then, for
any $\phi \in C_c(\cH_1(\alpha))$, for almost all $\theta \in [0,2\pi)$,
we have
\begin{equation}
\label{eq:birkhoff:flow}
\lim_{T \to \infty} \frac{1}{T} \int_0^T \phi(g_t r_\theta x) \, dt =
\int_\cM \phi \, d\nu_\cM, 
\end{equation}
where $\nu_\cM$ is the affine measure whose support is $\cM$. 
\end{theorem}

\begin{remark*}
\label{remark:birkhoff:flow:why}
The fact that (\ref{eq:birkhoff:flow}) holds for almost
all $x$ with respect to the Masur-Veech measure is an immediate consequence of
the Birkhoff ergodic theorem and the ergodicity of the Teichm\"uller
geodesic flow \cite{masur:interval}, \cite{veech:gauss}. The main point of
Theorem~\ref{theorem:birkhoff:flow} is that it gives a statement for
every flat surface $x$. This is important e.g. for applications to billiards in
rational polygons (since the set of flat surfaces one obtains from
unfolding rational polygons has Masur-Veech measure $0$). 
\end{remark*}

\begin{remark*}
\label{remark:birkhoff:flow:proofoutline}
The proof of Theorem~\ref{theorem:birkhoff:flow} is based on the
results of \cite{EM} and \cite{EMM} and the strong law of large numbers. One
complication is controlling the visits to neigborhoods of smaller
affine submanifolds,
which we do using the techniques of \cite{EMM}, \cite{Ath thesis}, 
\cite{Eskin:Masur} and which
were originally introduced by Margulis in \cite{Eskin:Margulis:Mozes:31}. 
\end{remark*}

\bold{The Kontsevich-Zorich cocycle.} We consider the \textit{Hodge bundle}
whose fiber above the point $(M,\omega)$ is $H^1(M,\reals)$. If we
choose a fundamental domain 
in Teichm\"{u}ller space
for the action of the mapping class group
$\Gamma$, then we have the cocycle $\tilde{A}: SL(2,\reals) \cross
\cH_1(\alpha) \to \Gamma$ where for $x$ in the fundamental domain,
$\tilde{A}(g,x)$ is the element of $\Gamma$ needed to return the point
$gx$ to the fundamental domain. Then, we define the Kontsevich-Zorich
cocycle $A(g,x)$ by 
\begin{displaymath}
A(g,x) = \rho(\tilde{A}(g,x)),
\end{displaymath}
where $\rho: \Gamma \to Sp(2g,\zed)$ is the homomorphism given by the
action on homology. The Kontsevich-Zorich cocyle can be interpreted as the
monodromy of the Gauss-Manin connection restricted to the orbit of
$SL(2,\reals)$, see e.g. \cite[page~64]{Zorich:survey}. 

\bold{A variant of the Oseledets multiplicative ergodic theorem.} 
\begin{theorem}
\label{theorem:oseledets:flow}
Fix $x \in \cH_1(\alpha)$, and let $\cM = \overline{SL(2,\reals)
  x}$ denote the smallest affine manifold containing $\cM$. Then
\begin{itemize}
\item[{\rm I.}] If $\psi_1(t,\theta) \le \dots \le
  \psi_{2g}(t,\theta)$  are the eigenvalues of the matrix $A^*(g_t,
  r_\theta x) A(g_t, r_\theta x)$, then for almost all $\theta \in
  [0,2\pi)$, we have 
\begin{equation}
\label{eq:oscelets:flow:I}
\lim_{t \to \infty} \frac{1}{t} \log \psi_i(t,\theta) = 2\lambda_i.
\end{equation}
Here the numbers $\lambda_1 \ge \dots \ge \lambda_{2g}$ depend only on
$\cM$. They are called the \textit{Lyapunov exponents} of the Kontsevich Zorich
cocycle on $\cM$.  
\item[{\rm II.}] For almost all $\theta$, the limit
\begin{displaymath}
\lim_{t \to \infty} (A^*(g_t, r_\theta x) A(g_t, r_\theta x))^{\frac{1}{2t}} \equiv
\Lambda(x,\theta) 
\end{displaymath}
exists. Moreover, the eigenvalues of the matrix $\Lambda(x,\theta)$,
taken with their multiplicities, coincide with the numbers
$e^{\lambda_i}$. Furthermore,
\begin{displaymath}
\lim_{t \to \infty} \frac{1}{t} \log \| A(g_t, r_\theta x)
\Lambda^{-t}(x,\theta) \| = \lim_{t \to \infty} \frac{1}{t} \log \|
\Lambda^n(x,\theta) A^{-1}(g_t, r_\theta x) \| = 0. 
\end{displaymath}
\item[{\rm III.}] Let $\alpha_1 < \dots < \alpha_s$ denote the
  distinct Lyapunov exponents $\lambda_i$. Let $\cU_i(x,\theta)
  \subset H^1(M,\reals)$
  denote the corresponding eigenspaces of $\Lambda(x,\theta)$. We set
  $\cV_0(x,\theta) = \{ 0 \}$ and $\cV_i(x,\theta) = \cU_1(x,\theta)
  \oplus \dots \oplus \cU_i(x,\theta)$. Then, for almost all $\theta$,
  and for any $v \in \cV_i(x,\theta) \setminus \cV_{i-1}(x,\theta)$,
  we have
\begin{displaymath}
\lim_{t \to \infty} \frac{1}{t} \log \|A(g_t, r_\theta x) v \| =
\alpha_i. 
\end{displaymath}
\end{itemize}
\end{theorem}

We note that Theorem~\ref{theorem:oseledets:flow} remains true if $A(
\cdot, \cdot)$ denotes the Kontsevich-Zorich cocycle acting on any
(continuous) subbundle of any tensor power of the Hodge bundle.

\begin{remark*} The fact that the conclusions of
  Theorem~\ref{theorem:oseledets:flow} hold for almost all $x$ with
  respect to the affine measure $\nu_\cM$ (or in particular with
  respect to the Masur-Veech measure) is just the classical Oseledets
  multiplicative ergodic theorem. The main point of
  Theorem~\ref{theorem:oseledets:flow} is that the conclusion holds
  for \textit{all} $x \in \cH_1(\alpha)$. This has some applications which partly motivated this paper,
  in particular in connection to the wind-tree model
  \cite{DHL:Windtree:diffusion}, \cite{Fraczek:Ulcigrai} and earlier results on IETs \cite{Zorich:deviation}, \cite{MMY}. See Section~\ref{sec:applications}.
  The question
  of whether Theorem~\ref{theorem:oseledets:flow} is true was raised
  in \cite{Forni:Sobolev}.
\end{remark*}

It is well known that parts II and III of
Theorem~\ref{theorem:oseledets:flow} follow from part I by an argument
which does not involve any ergodic theory (see
\cite{Goldsheid:Margulis}, 
from which our statement of the multiplicative
ergodic theorem was taken). It is thus enough to show that part I
holds for all $x$ and almost all $\theta$. Our proof of I is based on
the same ideas as the proof of Theorem~\ref{theorem:birkhoff:flow},
namely the results of \cite{EM}, \cite{EMM} and the strong law of
large numbers. However, we also need another important input: the
theorem of Filip \cite{Filip} 
stated as Theorem~\ref{theorem:filip} below. 
This complication can
be traced back to the fact that the Kingman and Oseledets ergodic
theorems can fail to hold at some points even for uniquely ergodic
systems (see \cite{Furman:osceledets} and references therein). 

\begin{definition}[$\nu$-measurable almost invariant splitting]
\label{def:almost:invariant:splitting}
Let $X$ be a space on which $G = SL(2,\reals)$ acts, preserving a
measure $\nu$. 
Suppose $V$ is a real vector space, and suppose $A: G \cross X \to
SL(V)$ is a cocycle. We say that $A$ has an almost invariant splitting 
if there exists $n > 1$ and for a.e $x$ there exist
nontrivial subspaces
$W_1(x), \dots, W_n(x) \subset V$ such that $W_i(x) \cap W_j(x) = \{0
\}$ for $1 \le i < j \le n$ and also for a.e $g \in G$ and
$\nu$-a.e. $x \in X$, 
\begin{displaymath}
A(g,x) W_i(x) = W_j(g x) \qquad \text{ for some $1 \le j \le n$ }. 
\end{displaymath}
The map $x \to \{ W_1(x), \dots, W_n(x) \}$ is required to be
$\nu$-measurable. 
\end{definition}

\begin{definition}[Strongly irreduclible cocycle]
\label{def:strongly:irreducible}
A cocycle $A$ is strongly irreducible with respect to the measure
$\nu$ if is does not admit a $\nu$-measurable almost invariant splitting. 
\end{definition}

In this paper, we prove the following:
\begin{theorem}
\label{theorem:lyapunov:top:flow}
Fix $x \in \cH_1(\alpha)$, and let $\cM = \overline{SL(2,\reals) x}$
be the smallest affine invariant manifold containing $x$. 
Let $V$ be $SL(2,\reals)$ invariant subbundle of (some exterior power
of) the Hodge bundle which is defined and \emph{is continuous} on
$\cM$. Let $A_V: SL(2,\reals) \cross \cM \to V$ denote the restriction
of (some exterior power of) 
the Kontsevich-Zorich cocycle to $V$, and suppose that $A_V$ 
is strongly irreducible with respect
to the affine measure $\nu_\cM$ whose support is $\cM$. Then, for
almost all $\theta \in [0,2 \pi)$, 
\begin{equation}
\label{eq:theorem:lyapunov:top:flow}
\lim_{t \to \infty} \frac{1}{t} \log \|A_V(g_t, r_\theta x) \| 
\end{equation}
exists and coincides with the top Lyapunov exponent of $A_V$. 
\end{theorem}

The main additional input needed for the proof of
Theorem~\ref{theorem:oseledets:flow} is the following:
\begin{theorem}[\protect{\cite{Filip}}]
\label{theorem:filip}
Let $A(\cdot, \cdot)$ denote (some exterior power of) 
the Kontsevich-Zorich cocycle restricted
to an affine invariant submanifold $\cM$. Let $\nu_\cM$ be the affine
measure whose support is $\cM$, and suppose $A$ has a
$\nu_\cM$-measurable almost-invariant splitting. Then, the subspaces
$W_i(x)$ in Definition~\ref{def:almost:invariant:splitting} can be
taken to depend continuously on $x \in \cM$. 
\end{theorem}

In fact it is proven in \cite{Filip} that the
  dependence of the $W_i(x)$ on $x$ is polynomial in the period coordinates.

\bold{Proof of Theorem~\ref{theorem:oseledets:flow} from
  Theorem~\ref{theorem:lyapunov:top:flow} and Theorem~\ref{theorem:filip}.}
Let $A(\cdot, \cdot)$ denote the Kontsevich-Zorich cocycle restricted
to an affine invariant submanifold $\cM$. Then by
\cite[Theorem~A.6]{EM} $A(\cdot, \cdot)$
is semisimple, in the sense that (after passing to some finite cover)
for $\nu_\cM$-almost all $x \in \cM$
there is a $\nu_\cM$-measurable direct sum decomposition 
\begin{displaymath}
H^1(M,\reals) = \bigoplus_{i=1}^n V_i(x),
\end{displaymath}
where all the subbundles $V_i$ are $\nu_\cM$-measurable,
$SL(2,\reals)$-invariant 
and strongly irreducible. By Theorem~\ref{theorem:filip}, the $V_i(x)$
can be taken to depend continuously on $x$. \mc{fix finite cover} Then, 
by Theorem~\ref{theorem:lyapunov:top:flow} it follows that the top Lyapunov
exponent on each $V_i$ is defined for almost all $\theta$. 
(To connect the conclusion of
Theorem~\ref{theorem:lyapunov:top:flow} with (\ref{eq:oscelets:flow:I}), 
note that the top
eigenvalue of $A_V(g_t, r_\theta x)^* A_V(g_t, r_\theta x)$ is
$\|A_V(g_t, r_\theta x ) \|^2$.)

To get that the rest of the Lyapunov exponents
are defined for almost all $\theta$ 
it suffices to repeat the argument for the cocycle acting on the
exterior powers of the Hodge bundle. (Note that the norm of
$A_V(g_t, r_\theta x)$ acting on $\bigwedge^d(V)$ is the product of the top
$d$ eigenvalues of $[A_V(g_t,r_\theta x)^* A_V(g_t r_\theta
x)]^{1/2}$ acting on $V$). 
This proves statement I of
Theorem~\ref{theorem:oseledets:flow}, and then statements II and III
of Theorem~\ref{theorem:oseledets:flow} follow as in
\cite{Goldsheid:Margulis}. 
\qed\medskip

\begin{remark}
\label{remark:2d:bundle}
For the case of a two-dimensional continuous subbundle $V$ of the Hodge
bundle, Theorem~\ref{theorem:oseledets:flow} follows from
Theorem~\ref{theorem:lyapunov:top:flow} (without the need for
Theorem~\ref{theorem:filip}). Indeed, by
\cite[Theorem~1.4]{Avila:Eskin:Moeller:yeti} any
$SL(2,\reals)$-invariant measurable subbundle of the Hodge bundle is
symplectic, and thus even dimensional. Thus, the restriction of
the cocycle to a two-dimensional subbundle is automatically strongly
irreducible. (This is the case which arises in \cite{DHL:Windtree:diffusion},
\cite{Fraczek:Ulcigrai}.) 
\end{remark}

\subsection{Outline of the proofs}
The main tool for the proof of the Birkhoff theorem for the random
walks is Lemma~\ref{lemma:limit:is:stationary} which is essentially \cite[Lemma
3.2]{BQIII} (and is sometimes called the ``Kesten law of large
numbers''). It shows that almost surely, the limit measure we are
considering is $\mu$-stationary. We then use the classification of
ergodic $\mu$-stationary measures in \cite[Theorem~1.6]{EM}. In our
setting, we do not initially know that the limiting measure is
ergodic. In fact, 
the primary additional complication in the
Birkhoff theorem, is to show that almost surely the limiting measure is a
probability measure that gives zero weight to any proper affine
$SL_2(\mathbb{R})$ invariant submanifold. This is treated by the now
standard technique of integral inequalties originating in
\cite[Section 5]{Eskin:Margulis:Mozes:31} (in our setting we appeal to
results of \cite{Ath thesis}, \cite{BQIII} and \cite{EMM}).

The stategy for our proof of the Oseledets theorem for the random
walk (or more precisely Theorem~\ref{theorem:lyapunov:top:walk} which
implies it) is to make the the standard version of the
Oseledets theorem effective in a large set, and then to use the
Birkhoff theorem we just proved to show that a typical trajectory will
spend most of the time in that large set. This turns out to be
sufficient to prove the upper bound in
Theorem~\ref{theorem:lyapunov:top:walk}, but not for the proof the lower
bound: see Remark \ref{rem:lower:bound}. This difficulty is dealt with
using Lemma~\ref{lemma:zero:one:law:star},
which is essentially \cite[Lemma 14.4]{EM}. For this argument to work,
we need the strong irreducibility assumption. 
 
The main step in the proof of the Birkhoff theorem for the flow in
\S{3} is Proposition~\ref{prop:ae:theta:P:invariant} (which replaces
Lemma~\ref{lemma:limit:is:stationary}): it implies that for almost all
$\theta$, the limit measure we have is $P$-invariant, where $P$ is the
upper triangular subgroup of $SL(2,\reals)$. We then use the
classification of ergodic $P$-invariant measures in \cite[Theorem~1.2]{EM};
the rest of the proof of the Birkhoff theorem for the flow is as in
the proof for the analogous theorem for random walk in \S{2}. 
The proof of Proposition~\ref{prop:ae:theta:P:invariant} (which
applies to any reasonable  action of $SL(2,\reals)$) is based on the
strong law of large numbers and Lemma \ref{lemma:ft:fs:correlation:bound}
which is essentially the fact that geodesic flow uniformly expands
small $SO(2)$ segments.  

We could have derived the Oseledets theorem for the flow
(Theorem~\ref{theorem:oseledets:flow}) from the Birkhoff theorem for
the flow Theorem~\ref{theorem:birkhoff:flow} as in \S{2}. We chose not
to do so since we show in \S{4} that it follows easily from the
Oseledets theorem for the random walk proved in \S{2}. 
 \subsection{Applications}\label{sec:applications}
Often the Birkhoff ergodic theorem or Oseledets multiplicative ergodic theorem have been inputs to prove dynamical theorems about flows on typical surfaces. By strengthening these results to hold on the typical direction on every surface many  of these results can be improved.
\subsubsection{The windtree model}
Following work of Ehrenfest and Ehrenfest \cite{windtree}, Hardy and Weber \cite{windtree2} introduced  a model which has attracted attention recently. One places a fixed $[0,a] \times [0,b]$ rectangular scatter $\mathbb{Z}^2$ periodically in the plane with sides parallel to the axes and examines the behavior of a point mass traveling in this space which follows the rules of elastic collision when it hits one of the rectangles. Let $T(a,b)$ be the complement of the obstacles in the plane. Let $\phi^{\theta}$ be the flow in direction $\theta$ on $T(a,b)$.
\begin{thm}\cite{DHL:Windtree:diffusion} Let $d(\cdot,\cdot)$ denote Euclidean distance in the plane. 
For almost every $(a,b)$ and almost every $\theta$ and every point $p \in T(a,b)$ we have 
$$ \underset{T \to \infty}{\limsup}\, \frac{\log(d(p,\phi^{\theta}_T(p))}{\log(T)}=\frac 2 3.$$
\end{thm}
  
  In subsequent versions of \cite{DHL:Windtree:diffusion}, Delecroix, Hubert and Leli\'evre used Theorem~\ref{theorem:oseledets:flow} to strengthen this theorem to apply to all obstacles. That is, for all $(a,b)$ and almost every $\theta$ and every point $p \in T(a,b)$ we have
  $$ \underset{T \to \infty}{\limsup}\, \frac{\log(d(p,\phi^{\theta}_T(p))}{\log(T)}=\frac 2 3.$$
  
  \begin{thm}\label{Fraczek:Ulcigrai} \cite{Fraczek:Ulcigrai} For almost every $(a,b)$ and almost every $\theta$, $\phi^{\theta}$ is not ergodic with respect to Lebesgue measure. In fact, the ergodic decomposition of Lebesgue measure has uncountably many ergodic components.
  \end{thm}

  In \cite{Fraczek:Ulcigrai:preprint} using Theorem \ref{theorem:birkhoff:flow} and \ref{theorem:oseledets:flow} Fraczek and Ulcigrai extended Theorem \ref{Fraczek:Ulcigrai} to apply to all $(a,b)$.

  \subsubsection{Behavior of Birkhoff sums and averages} In a seminal work, A. Zorich connected the deviation of ergodic averages of certain functions to the behavior of the Kontsevich-Zorich cocycle \cite{Zorich:deviation}. In doing so he proved estimates on the error term in the Birkhoff ergodic theorem for these functions for the flow in almost every direction on almost every flat surface. Theorem \ref{theorem:oseledets:flow} shows that these results hold for the flow in the typical direction on every surface. Zorich's insight has been developed by other authors and we now mention some results that can be improved by Theorem \ref{theorem:oseledets:flow}.
  
  Marmi, Moussa and Yoccoz defined Roth type interval exchange transformations. We state their condition in the language of this paper. 
  We first need a preliminary discussion. To every surface $x$ in our fundamental domain we can associate a canonical basis of cohomology. In this way we have a basis for $\tilde{A}(g_t,x)g_tx$. Call this $\mathcal{B}_t=\{b_1,...,b_{2g}\}$. We also have the basis for $x$, $\mathcal{C}=\{c_1,...,c_{2g}\}$ and its image under $g_t$ which is $\{g_tc_1,...,g_tc_{2g}\}$. Chose $t_1(x)$ to be the smallest time so that for any $t\geq t_1$ we have that for each $b_i$, in the unique expression $b_i=\sum a_ig_tc_i$ all the $a_j\neq 0$. Inductively choose $t_{i+1}=t_i+t_1(g_{t_i}x)$. 
  \begin{defin} The vertical flow on a surface $x$ is \emph{Roth type} if
  \begin{enumerate}
  \item[(a)] $\underset{i \to \infty}{\lim}\, \frac{t_i-t_{i-1}}{t_i}=0$.
  \item[(b)] Let $V_0(x)$ denote the set of vectors $\bar{v}$ so that $\sum v_ih_i=0$ where $h_i$ are the horizontal components of the periods of $x$. There exists $C\in \mathbb{R}$, $\theta<1$ so that $||A_{V_0}(g_t,x)||<Ct^{\theta}$ where $A_{V_0}$ is the restriction of $A$ to the vector space $V_0$.
  \item[(c)] Let $V^{(s)}$ denote the set of vectors, $v$, so that there exists $\sigma>0$, $C>0$ so that $\|A(g_t,x)v\|<Ct^{-\sigma}$ for all $t>1$.  Let $B(a,b)$ be the operator that sends $V^{(s)}(g_ax)$ to $V^{(s)}(g_bx)$ and $B_*(a,b)$ be the operator that sends $H^1(M,\mathbb{R})/V^{(s)}(g_ax)$ to 
  $H^1(M,\mathbb{R})/V^{(s)}(g_bx)$. For all $\tau>0$ we have that there exists $C_{\tau,x}:=C$ so that 
  $$\|B(t_i,t_{i+r})\|\leq C\|B(0,n)\|^{\tau} \text{ and } \|B_*(t_i,t_{i+r})^{-1}\|<C\|B(0,t_{i+r})\|.$$ 
  \end{enumerate}
  \end{defin}
  \begin{thm}\cite{MMY}\label{thm:mmy} Let $T$ be an IET that satisfies the Keane condition and which arises as first return to a transversal of a vertical flow that is Roth type. Let $f$ be a bounded variation function. Then there exists a function $\phi$ that is constant on each subinterval of $T$  and another function $\psi$ so that $f-\phi=\psi-\psi \circ T$.
  \end{thm}
  Suppose we fix a flat surface and a tranversal. We now consider the
first return to the transversal of the flow in the direction
$\theta$. This gives a family of IETs parametrized by $\theta$.  We
claim that Theorems  \ref{theorem:birkhoff:flow} and
\ref{theorem:oseledets:flow} extend Theorem \ref{thm:mmy} to apply to
a full measure subset of this one parameter family of IETs (for any
choice of flat surface and transversal). The most involved step is verifying (a).
   Let $G_N=\{x:t_1(x)<N\}$.
   \begin{lemma}\label{lem:in compact} For any $SL(2,\mathbb{R})$ ergodic measure  $\nu$ there exists $N_{\nu}$ so that $G_{N_{\nu}}$ contains an 
   open set of  positive $\nu$-measure.
   \end{lemma}
      \begin{proof}By  \cite[Theorem 2 (iii)]{KW} for every flat surface $x$ there exists $\theta\in S^1$ and a compact set $K'$  so that $g_tr_\theta x \in K'$ for all $t$. By Masur's critertion for unique ergodicity, if $K$ is any compact set then there exists $N_K$ so that $g_tx \in K$ for all $0\leq t\leq N_K$ implies $x \in G_N$.  Choosing a compact set $\mathcal{K}$ so that $K'$ is a subset of the interior of  $\mathcal{K}$ we have that for every $M\in \mathbb{R}$ there is an open neighborhood $U_M$ of $r_\theta x$ so that $g_tU_M \subset \mathcal{K}$ for all $0\leq t \leq M$. Choose $M$ so that if $g_ty \in \mathcal{K}$ for all $0\leq t \leq M$ then $y \in G_M$. Thus, $U_M$ is an open set in $G_M$ with positive $\nu$ measure where $\nu$ is the unique $SL(2,\mathbb{R})$-ergodic probability measure with support $\overline{SL(2,\mathbb{R})x}$.
   \end{proof}
   \begin{lemma} If Theorem \ref{theorem:birkhoff:flow} and Lemma \ref{lem:in compact} are true then for every $x$,
   $$\{\theta: \text{ so that condition (a) is satisfied for }r_\theta x\}$$ has full Lebesgue measure.  
   \end{lemma}
   \begin{proof} Let $x$ be given and $\nu$ be the unique $SL(2,\mathbb{R})$-ergodic probability measure whose support is $\overline{SL(2,\mathbb{R})x}$. By Lemma \ref{lem:in compact}  there exists $N$ and $U \subset G_N$ open with positive $\nu$-measure. By Theorem \ref{theorem:birkhoff:flow} we have
    that $\underset{T \to \infty}{\lim} \, \frac 1 T \int_0^T \chi_U(g_tr_{\theta}x)dt$ exists and is greater than zero for almost every $\theta$. Now, $\int_0^{t_i(r_\theta x)+N} \chi_U(g_tr_\theta x)dt \geq \int_0^{t_{i+1}(r_\theta x)-N} \chi_U(g_tr_{\theta}x)dt$. Indeed, if there exists $s>t_{i}(r_{\theta}x)+N$ so that $g_sr_\theta x \in U \subset G_N$ then $t_{i+1}(r_\theta x)\leq s+N$. So the absence of condition (a) for $r_{\theta}x$ implies that there exists a fixed $c>0$ so that there are infinitely many $i$ with $(1+c)t_i<t_{i+1}$ and so
    \begin{multline}
    \frac 1 {t_i(r_\theta x)}(\int_0^{t_i (r_\theta x)} \chi_U(g_tr_\theta x)dt+N)\geq \frac 1 {t_i(r_\theta x)}
    (\int_0^{t_{i+1} (r_\theta x)} \chi_U(g_tr_\theta x)dt-N) >\\
    (1+c) \left(\frac 1 {t_{i+1}(r_\theta x)}( \int_0^{t_{i+1}(r_\theta x)} \chi_U (g_t r_\theta x)dt-N)\right).
    \end{multline}
Since $t_i$ goes to infinity with $i$ this implies that Theorem \ref{theorem:birkhoff:flow} does not hold for $r_\theta x$.
    So the absence of (a) for a positive measure set of directions contradicts Theorem \ref{theorem:birkhoff:flow}.
   \end{proof}

  This completes the verification of (a). Condition (b) follows from the the fact that $\lambda_2<1$ \cite[Theorem 0.2 (i)]{Forni:Deviation}. Condition (c) holds by Theorem \ref{theorem:oseledets:flow}. See Section \cite[Section 4.5]{MMY} for how (b) and (c) follow from the conclusion of Oseledets theorem (with additional input for (b)).
  
  In a recent further development a stronger notion, \emph{restricted Roth type} has been introduced to study the regularity of solutions to the cohomological equation \cite{MY}. Theorem \ref{theorem:oseledets:flow} shows that every flat surface whose orbit closure has $\lambda_g>0$ has that the flow in almost every direction is restricted Roth type.
  
  It is likely that Theorems 1 and 2 of \cite{buf} can be extended via Theorem \ref{theorem:oseledets:flow}.
\color{black}

    \bold{Acknowledgments:} We thank the Alexander Bufetov and Corinna Ulcigrai for suggestions that improved the readability of the paper. We thank Ronggang Shi for pointing out a mistake. J. Chaika was supported in part by NSF grants DMS 1300550 and 1004372.
    A. Eskin was supported in part by NSF grant DMS 1201422 and the Simons Foundation.



\section{Random walks}
\label{sec:random:walks}
To provide intuition, we first prove versions of
Theorem~\ref{theorem:birkhoff:flow} and
Theorem~\ref{theorem:lyapunov:top:flow} for random walks.
We use the following setup. Let $\mu$ be an $SO(2)$-bi-invariant
compactly supported measure on $SL(2,\reals)$ which is absolutely
continuous with respect to Haar measure. We consider the random walk
on $SL(2,\reals)$ whose transition probabilities are given by
$\mu$. This also defines a random walk on $\cH_1(\alpha)$, via the
$SL(2,\reals)$ action. (The trajectories of this random walk stay in
Teichm\"uller disks.) 

Let $\bar{g} = (g_1, \dots, g_2, \dots, )$ denote an element of
$SL(2,\reals)^\natls$. Let $\mu^\natls$ denote the product measure on
$SL(2,\reals)^\natls$. It follows from  the Oseledets multiplicative
ergodic theorem  that for $\mu^\natls$-almost-all $\bar{g}$, the trajectory 
\begin{displaymath}
g_1, g_2 g_1, \dots, g_{n-1} \dots g_1, \ g_n g_{n-1} \dots g_1
\end{displaymath}
tracks, up to sublinear error,  
a geodesic of the form $\{ g_t r_\theta \st t \in \reals \}$ with
respect the the \textit{right-invariant} metric on $SL(2,\reals)$.  
(This will be made more precise in \S\ref{sec:proof:osceledts:flow}.) The angle
$\theta$ depends on $\bar{g}$, but as we show in
\S\ref{sec:proof:osceledts:flow}, 
the distribution of $\theta$'s 
induced by $\mu^\natls$ is uniform. Thus, we expect to have analogues
of Theorem~\ref{theorem:birkhoff:flow} and
Theorem~\ref{theorem:oseledets:flow} (and
Theorem~\ref{theorem:lyapunov:top:flow}) in the random walk setup, where
the clause ``for almost all $\theta$'' is replaced by the clause ``for
almost all $\bar{g}$''. This is indeed the case, and we find the proofs
of the random walk versions, namely Theorem~\ref{theorem:birkhoff:walk} 
and Theorem~\ref{theorem:lyapunov:top:walk} a bit cleaner and easier
to follow. Also we will see below that
Theorem~\ref{theorem:lyapunov:top:flow} follows formally  from its random
walk version Theorem~\ref{theorem:lyapunov:top:walk}.

\subsection{A Birkhoff type theorem for the random walk}
\begin{theorem}
\label{theorem:birkhoff:walk}
Suppose $x \in \cH_1(\alpha)$. Let $\cM = \overline{SL(2,\reals) x}$
be the smallest affine invariant manifold containing $x$. Then, for
any $\phi \in C_c(\cH_1(\alpha))$, for $\mu^\natls$-almost all $\bar{g} \in
SL(2,\reals)^\natls$, we have
\begin{equation}
\label{eq:birkhoff:walk}
\lim_{N \to \infty} \frac{1}{N} \sum_{n=1}^N \phi(g_n \dots g_1 x) =
\int_\cM \phi \, d\nu_\cM, 
\end{equation}
where $\nu_\cM$ is the affine measure whose support is $\cM$. 
\end{theorem}

\begin{corollary}
\label{cor:hits:often}
Suppose $x \in \cH_1(\alpha)$. Let $\cM = \overline{SL(2,\reals) x}$
be the smallest affine invariant manifold containing $x$. Let $U$ be
an open subset of $\cM$. 
Then, for $\mu^\natls$-almost all $\bar{g} \in
SL(2,\reals)^\natls$, we have
\begin{displaymath}
\lim_{N \to \infty} \frac{1}{N} \sum_{n=1}^N \chi_U(g_n \dots g_1 x) \geq 
\nu_\cM(U), 
\end{displaymath}
 and equality holds if $\nu_\cM(\partial U)=0$, 
where $\nu_\cM$ is the affine measure whose support is $\cM$. 
\end{corollary}

Our proof of Theorem~\ref{theorem:birkhoff:walk} follows
\cite{BQIII}. Let $x$, $\cM$ and $\nu_\cM$ be as in
Theorem~\ref{theorem:birkhoff:walk}. 
We begin with the following:
\begin{lemma}
\label{lemma:limit:is:stationary}
For every $x$, almost every $\bar{g}\in SL(2,\mathbb{R})^{\mathbb{N}}$, 
if $\tilde{\nu}$ is a weak-* limit point of 
$$ \frac 1 N \sum_{n=1}^N \delta_{g_n \dots g_1 x} $$ 
then $\tilde{\nu}$ is $\mu$-stationary (i.e. $\mu \ast \tilde{\nu} =
\tilde{\nu}$). 
\end{lemma}

\bold{Proof.}  It suffices to check a countable subset of $C_c(\cM)$,
so it suffices to have the result for each fixed function in
$C_c(\cM)$.
We follow \cite[Lemma~3.2]{BQIII}. 
Let $\phi \in C_c(\cM)$ be a test function. 
Let 
\begin{displaymath}
f_n(x,\bar{g})=\phi(g_n g_{n-1} \dots g_1 x)-\int_{SL(2,\mathbb{R})}
\phi(h g_{n-1} \dots g_1 x) \, d\mu(h). 
\end{displaymath}
By definition $\int_{f_1\in A_1,...,f_{n-1}\in A_n-1}
f_n(x,\bar{g})d\mu^{\mathbb{N}}=0$ for any $n \in \natls$  \mc{fix
  notation} and any 
subsets $A_1,...,A_{n-1}$ of $\reals$. 
Additionally, $\|f_n\|_{\infty} \leq 2\|\phi\|_{\infty}$. 
So by the strong law of large numbers 
\begin{displaymath}
\underset{N \to \infty}{\lim} \frac{1}{N} \sum_{n=1}^N f_n(x,\bar{g})=0
\qquad \text{ for a.e. $\bar{g}$.}
\end{displaymath}
Thus $\tilde{\nu}$ is $\mu$-stationary almost everywhere.
\qed\medskip

{Lemma~\ref{lemma:limit:is:stationary} will allow us to use the classification of ergodic stationary probability measures in \cite[Theorem~1.6]{EM}. However, the limit measures of Lemma~\ref{lemma:limit:is:stationary} may not be probability measures, and may not be ergodic. To deal with this, we will use the following, }
\color{black}
(which is the main technical result of
\cite{EMM}):
\begin{proposition}[\protect{see \cite[Proposition~2.13]{EMM},
  \cite[Lemma~3.2]{EMM}}]
\label{prop:EMM:main:proposition}
Let $\cN \subset \cH_1(\alpha)$ be an affine submanifold. (In this
proposition $\cN = \emptyset$ is allowed). Then there
exists an $SO(2)$-invariant 
function $f_\cN: \cH_1(\alpha) \to [1,\infty]$ with the following
properties:
\begin{itemize}
\item[{\rm (a)}] $f_\cN(x) = \infty$ if and only if $x \in \cN$, and
  $f_\cN$ is bounded on compact subsets of $\cH_1(\alpha)\setminus \cN$. 
   For any $\ell > 0$, the set $\overline{\{ x \st f(x) \le \ell \}}$ is a
  compact subset of $\cH_1(\alpha)\setminus \cN$. 
\item[{\rm (b)}] There exists  $b > 0$ (depending on $\cN$) and for
  every $0 < c < 1$   there exists $n_0 > 0$ (depending on $\cN$ and
  $c$) such that for all $x \in \cH_1(\alpha)$ and all $n > n_0$,
\begin{displaymath}
\int_{SL(2,\reals)} f_\cN(x) \, d\mu^{(n)}(x) \le c f_\cN(x) + b.
\end{displaymath}
Here $\mu^{(n)}$ denotes the convolution $\mu \ast \dots \ast \mu$ ($n$
times). 
\item[{\rm (c)}] There exists $\sigma > 1$ such that for all $g \in
  SL(2,\reals)$ with $\|g \| \le 1$ and all $x \in \cH_1(\alpha)$, 
\begin{displaymath}
\sigma^{-1} f_\cN(x) \le f_\cN(g x) \le \sigma f_\cN(x).
\end{displaymath}
\end{itemize}
\end{proposition}
 The next lemma is a formal consequence of the previous system of inequalities. For it's derivation from Proposition~\ref{prop:EMM:main:proposition}, see the (self-contained) proof of \cite[Proposition~3.9]{BQIII}. It
is helpful for our purposes because it can apply to every $x$.
\color{black}
\begin{lemma}[\protect{\cite[Proposition 3.9]{BQIII}}]
\label{lemma:control:values}
Suppose $f_\cN$ is a function satisfying the conditions of
Proposition~\ref{prop:EMM:main:proposition}. Then, for any $0 < c <1$ any 
$M > 0$ and
$\mu^{\natls}$-almost-all $\bar{g} \in SL(2,\reals)^\natls$, we have,
for all sufficiently large $n$, 
\begin{equation}
\label{eq:control:values}
\frac{1}{n} |\{0<k<n: f_\cN(g_k....g_1 x)>M\}|\leq
\frac{C}{(1-c)M},
\end{equation}
where $C$ depends only on the constants $n_0$, $b$ and $\sigma$ of
Proposition~\ref{prop:EMM:main:proposition}. 
\end{lemma}

\bold{Proof of Theorem~\ref{theorem:birkhoff:walk}.}
Let $\tilde{\nu}$ be any weak-* limit point of $ \frac 1 N
\sum_{n=1}^N \delta_{g_n \dots g_1 x}$. By
Lemma~\ref{lemma:limit:is:stationary}, for almost all $\bar{g}$,
$\tilde{\nu}$ is $\mu$-stationary. By \cite[Theorem~1.6]{EM}, any
$\mu$-stationary measure (such as $\tilde{\nu}$) is $SL(2,\reals)$-invariant. 

By \cite[Theorem~1.4]{EM}, any ergodic $SL(2,\reals)$-invariant
measure is affine. Therefore, since $\tilde{\nu}$ is supported on
$\cM$, $\tilde{\nu}$ has can be decomposed into ergodic components as
\begin{displaymath}
\tilde{\nu} = \sum_{\cN \subseteq \cM} a_\cN \, \nu_\cN,
\end{displaymath}
where $a_\cN \in [0,1]$ and 
the sum is over the affine invariant submanifolds $\cN$ contained in
$\cM$. (Here $\cN = \cM$ is allowed). By \cite[Proposition~2.16]{EMM}
this is a countable sum. By applying
(\ref{eq:control:values}) for the case $\cN = \emptyset$ we get that
for $\mu^\natls$-almost all $\bar{g}$, 
$\tilde{\nu}$ is a probability measure. Then, by applying
(\ref{eq:control:values}) again with $\cN$  any affine invariant
submanifold properly contained in $\cM$,  we see that for
$\mu^\natls$-almost-all $\bar{g}$, 
$\tilde{\nu}(\cN) = 0$. Thus $a_\cN = 0$ for $\cN$ 
properly contained in $\cM$. Since $\tilde{\nu}$ is a probability
measure, this forces $\tilde{\nu} = \nu_\cM$, completing the proof of
Theorem~\ref{theorem:birkhoff:walk}.  
\qed\medskip

\subsection{An Oseledets type theorem for the random walk.}
\begin{theorem}
\label{theorem:lyapunov:top:walk}
Fix $x \in \cH_1(\alpha)$, and let $\cM = \overline{SL(2,\reals) x}$
be the smallest affine invariant manifold containing $x$. 
Let $V$ be $SL(2,\reals)$ invariant subbundle of (some exterior power
of) the Hodge bundle which is defined and \emph{is continuous} on
$\cM$. Let $A_V: SL(2,\reals) \cross \cM \to V$ denote the restriction
of (some exterior power of) 
the Kontsevich-Zorich cocycle to $V$, and suppose that $A_V$ 
is strongly irreducible with respect
to the affine measure $\nu_\cM$ whose support is $\cM$. Then, for
$\mu^\natls$-almost-all $\bar{g} = (g_1, \dots, g_n, \dots)$, 
\begin{equation}
\label{eq:lyapunov:top:walk}
\lim_{n \to \infty} \frac{1}{n} \log \|A_V(g_n \dots g_1 , x) \| = \lambda_1
\end{equation}
where $\lambda_1$ is the top Lyapunov exponent of $A_V$ restricted to
$\cM$ (and depends only on $\mu$, $V$ and $\cM$). 
\end{theorem}

Let $m = \dim(V)$. 
We recall the statement of the Oseledets multiplicative ergodic
theorem from e.g. \cite{Goldsheid:Margulis} in this setting:
\begin{theorem}
\label{theorem:osceledets:ae:walk}
For \emph{$\nu_\cM$-almost
  all} $y \in \cM$ and $\mu^\natls$-almost-all $\bar{g} \in
SL(2,\reals)^\natls$, the following hold:
\begin{itemize}
\item[{\rm I.}] Let $\psi_1(n,\bar{g},y) \le \dots \le
  \psi_m(n,\bar{g},y)$  denote the eigenvalues of the matrix 
$$A_V^*(g_n
  \dots g_1,y) A_V(g_n \dots g_1, y).$$ 
Then for $1 \le i \le m$,
\begin{equation}
\label{eq:oscelets:flow:I}
\lim_{n \to \infty} \frac{1}{t} \log \psi_i(n,\bar{g},y) = 2\lambda_i.
\end{equation}
Here the numbers $\lambda_1 \ge \dots \ge \lambda_{m}$ depend only on
$\nu_\cM$ and $V$.  They are the \textit{Lyapunov exponents} of the cocycle
$A_V$ on $\cM$.  
\item[{\rm II.}] The limit
\begin{displaymath}
\lim_{n \to \infty} (A_V^*(g_n \dots g_1, y) A_V(g_n \dots g_1,
y))^{\frac{1}{2n}} \equiv 
\Lambda(y,\bar{g}) 
\end{displaymath}
exists. Moreover, the eigenvalues of the matrix $\Lambda(y,\bar{g})$,
taken with their multiplicities, coincide with the numbers
$e^{\lambda_i}$. Furthermore,
\begin{multline}
\label{eq:general:sublinear:tracking}
\lim_{n \to \infty} \frac{1}{n} \log \| A_V(g_n \dots g_1, y)
\Lambda^{-n}(y,\bar{g}) \| = \\ = \lim_{n \to \infty} \frac{1}{n}\log \|
\Lambda^n(y,\bar{g}) A_V^{-1}(g_n \dots g_1, y) \| = 0. 
\end{multline}
\item[{\rm III.}] Let $\alpha_1 < \dots < \alpha_s$ denote the
  distinct Lyapunov exponents $\lambda_i$. Let $\cU_i(y,\bar{g})
  \subset V$
  denote the corresponding eigenspaces of $\Lambda(y,\bar{g})$. We set
  $\cV_0(y,\bar{g}) = \{ 0 \}$ and $\cV_i(y,\bar{g}) = \cU_1(y,\bar{g})
  \oplus \dots \oplus \cU_i(y,\bar{g})$. Then, for almost all $y,
  \bar{g}$, 
  and for any $v \in \cV_i(y,\bar{g}) \setminus \cV_{i-1}(y,\bar{g})$,
  we have
\begin{displaymath}
\lim_{n \to \infty} \frac{1}{n} \log \|A_V(g_n \dots g_1, y) v \| =
\alpha_i. 
\end{displaymath}
\end{itemize}
\end{theorem}

\begin{remark}
As was done \S\ref{sec:intro}, one can use the  Filip's 
Theorem (Theorem \ref{theorem:filip}) and
Theorem~\ref{theorem:lyapunov:top:walk} to show that the conclusions
of Theorem~\ref{theorem:osceledets:ae:walk} hold for all $y$ (and
almost all $\bar{g}$) provided
$\cM$ is the smallest affine invariant manifold containing $y$ (or equivalently 
$\cM = \overline{SL(2,R) y}$). 
\end{remark}

One can use strong irreducibility to show:
\begin{lemma}\label{lemma:noneffective} For almost every $x$, every $v$ and almost every $\bar{g}$ we have 
$$\underset{n \to \infty}{\lim}\, \frac 1 n||A_V(g_n\dots g_1,x)v||=\lambda_1.$$
\end{lemma}
In order to prove theorem it is important to make this effective:
\color{black}

\bold{Notation.} For $L \in \natls$, let $\mu^L$ denote the measure on
$SL(2,\reals)^L$ given by $\mu \cross \mu \cross \dots \cross \mu$ ($L$
times).

\bold{The set $E_{good}(\epsilon,L)$.}
Suppose $\epsilon > 0$, $L \in \natls$. Let
$E_{good}(\epsilon,L)$ denote the set of $y \in \cM$ such that for each $v \in
V$ there exists a subset $H(v) \subset SL(2,\reals)^L$ so that
\begin{equation}
\label{eq:measure:Hv}
\mu^L(H(v)) > 1-\epsilon,
\end{equation}
and for all $(h_1,\dots, h_L) \in H(v)$,
\begin{equation}
\label{eq:def:Hv}
(\lambda_1 - \epsilon)^L \le \frac{\|A_V(h_L \dots h_1,y) v\|}{\|v\|}
\le \|A_V(h_L \dots h_1,y) \| \le (\lambda_1 + \epsilon)^L.
\end{equation}

\begin{remark} \label{rem:lower:bound} Observe the order of quantifiers. For any $y$ in $E_{good}(\epsilon,L)$ and any $v$ most $\bar{h}$ satisfy Equation (\ref{eq:def:Hv}). This is convenient to apply the strong law of large numbers to obtain the lower bound in the proof of Theorem \ref{theorem:lyapunov:top:walk}. In general situations the concern is that the largest eigenvector of $A_V(g_n...g_1,x)$ is contracted by $A_V(h_L...h_1,g_n...g_1x)$ for most $\bar{h}$. The order of  
  quantifiers removes this concern because whatever the direction of the largest eigenvector of $A_V(g_n...g_1,y)$ is, if $g_n...g_1y \in E_{good}(\epsilon,L)$, then for most $\bar{h}$ it will be expanded by roughly the right amount in the next $L$ steps.

The condition that it holds for every $v$ is also what is non-trivial about the next lemma (and Lemma \ref{lemma:noneffective}). Otherwise it would just follow formally from making the Oseledets theorem effective. 
\end{remark}
\color{black}

The following Lemma is a key step in our proof.
\begin{lemma}
\label{lemma:Egood:L}
For any fixed $\epsilon > 0$,
\begin{displaymath}
\lim_{L \to \infty} \nu_{\cM}(E_{good}(\epsilon,L)) = 1.
\end{displaymath}
\end{lemma}
We choose to prove this lemma directly rather than proving Lemma \ref{lemma:noneffective} and deducing it as a formal consequence of  making the statement of Lemma \ref{lemma:noneffective} effective.
\color{black}

\subsection{Proof of Lemma \ref{lemma:Egood:L}}


Fix $1 \le s \le
m$, and let $Gr_s(V)$ denote the Grassmanian of $s$-dimensional subspaces
in $V$. Let $\hat{\cM} = \cM \cross Gr_s(V)$. We then have an
action of $SL(2,\reals)$ on $\hat{\cM}$, by 
\begin{displaymath}
g \cdot (x, W) = (gx, A_V(g,x) W). 
\end{displaymath}
Let $\hat{\nu}_\cM$ be a  $\mu$-stationary measure on $\hat{\cM}$
which
projects to $\nu_\cM$ under the natural map $\hat{\cM} \to \cM$ where $\nu_\cM$ is ergodic.
We may write
\begin{displaymath}
d\hat{\nu}_\cM (x,U) = d\nu_\cM(x) \, d\eta_x(U),
\end{displaymath}
where $\eta_x$ is a measure on $Gr_s(V)$. 

For $v \in V$, let 
\begin{displaymath}
I(v) = \{ U \in Gr_s(V) \st v \in U\}.
\end{displaymath}

\begin{lemma}[\protect{\cite[Lemma C.10(i)]{EM}, \cite[Lemma 4.1]{EMath}}]
\label{lemma:non:atomic:lyapunov}
Suppose the cocycle $A_V$ is strongly irreducible with respect to $\nu_\cM$.  
Then for almost all $y \in \cM$, for any $v_y \in V$, $\eta_y(I(v_y)) = 0$. 
\end{lemma}

\bold{Proof.} The proof is given in \cite[Appendix~C]{EM}. The
essential idea is that if 
conclusion of Lemma~\ref{lemma:non:atomic:lyapunov} is false, then the
cocycle would have to permute some finite collection of subspaces,
contradicting the strong irreducibility assumption. It is not necessary to assume that $\hat{\nu}_{\cM}$ is ergodic so long as it is stationary and its projection to the base is invariant and ergodic.
\qed\medskip

The following lemma is a consequence of Lemma~\ref{lemma:non:atomic:lyapunov}: 
\begin{lemma}[\protect{\cite[Lemma~14.4]{EM}}]
\label{lemma:zero:one:law:star}
For every $\delta > 0$ and every $\epsilon > 0$ 
there exists $E_{good} \subset \cM$ with 
$\nu_\cM(E_{good}) > 1 - \delta$ and $\sigma =
\sigma(\delta,\epsilon) > 0$, 
such that for any $y \in E_{good}$ and any vector $w \in \proj^1(V)$, 
\begin{equation}
\label{eq:random:subspace}
\mu^\natls\left( \{\bar{g} \ \st d(w,\cV_{s-1}(y,\bar{g})) > \sigma \}
\right) > 1-\epsilon. 
\end{equation}
\end{lemma}
(In (\ref{eq:random:subspace}), 
$d( \cdot, \cdot)$ is some distance on the projective space
$\proj^1(V)$).

\bold{Proof.} We reproduce the proof from \cite[Lemma~14.4]{EM} for
the convenience of the reader. 
For $F \subset Gr_{s-1}(V))$ (the Grassmanian of
$s-1$ dimensional subspaces of $V$) let
\begin{displaymath}
\hat{\nu}_x(F) = \mu^\natls \left( \{ \bar{g} \in SL(2,\reals)^\natls \st
  \cV_{s-1}(y,\bar{g}) \in F \} \right),
\end{displaymath}
and let $\hat{\nu}$ denote the measure on the bundle $\cM \cross
Gr_{s-1}(V)$ given by  
\begin{displaymath}
d\hat{\nu}(x,W) = d\nu_\cM(x) \, d\hat{\nu}_x(W). 
\end{displaymath}
Then, $\hat{\nu}$ is a stationary measure for
the random walk. 
Let 
\begin{displaymath}
Z = \{ y \in \cM \st  \hat{\nu}_y(I(w)) > 0 \text{ for some  $w
  \in \proj^1(V)$} \}.
\end{displaymath}
 However, if $\nu_{\mathcal{M}}(Z)>0$  this contradicts
Lemma~\ref{lemma:non:atomic:lyapunov}, since 
the action of the cocycle on $V$ is strongly
irreducible. 
 Thus, $\nu(Z) = 0$ and $\nu(Z^c) = 1$. By definition, 
for all $y \in Z^c$ and all $w_y \in V$, 
\begin{displaymath}
\mu^\natls\left( \{ \bar{g} \in SL(2,\reals)^\natls 
\st w_y \in \cV_{s-1}(y,\bar{g}) \} \right) = 0. 
\end{displaymath}
Fix $y \in Z^c$. Then, for every $w_y \in \proj^1(V)$ there
exists $\sigma_0 = \sigma_0(y,w_y,\epsilon) > 0$ such that 
\begin{displaymath}
\mu^\natls\left( \{ \bar{g} \in SL(2,\reals)^\natls \st
  d(\cV_{s-1}(y,\bar{g}),w_y) > 
  2\sigma_0(y,w_y,\epsilon) \} 
\right) > 1 - \epsilon.  
\end{displaymath}
Let $U(y,w) = \{ z \in \proj^1(V) \st d(z,w) < \sigma_0(y,w,\epsilon)
\}$. Then the $\{U(y,w)\}_{w \in \proj^1(V)}$ form an open
cover of the compact space $\proj^1(V)$, and therefore there
exist $w_1, \dots w_n$ with
$\proj^1(V) = \bigcup_{i=1}^n U(y,w_i)$.
Let $\sigma_1(y,\epsilon) = \min_i \sigma_0(y,w_i,\epsilon)$. Then, for
all $y \in Z^c$, 
\begin{displaymath}
\mu^\natls\left( \{ \bar{g} \in SL(2,\reals)^\natls \st
  d(\cV_{s-1}(y,\bar{g}),w) > 
  \sigma_1(y,\epsilon) \} \right) > 1 - \epsilon.  
\end{displaymath}
Let $E_N(\epsilon) = \{ x \in Z^c \st \sigma_1(x,\epsilon) > \frac{1}{N}\}$. 
Since $\bigcup_{N=1}^\infty E_N(\epsilon) = Z^c$ and 
$\nu(Z^c) = 1$, there exists $N = N(\delta,\epsilon)$ such that
$\nu(E_N(\epsilon)) > 1-\delta$. Let $\sigma = 1/N$ and let $E_{good} =
E_N$. 
\qed\medskip

\color{black}
Let $\cU_i(n,y,\bar{g})$ denote the direct sum of the eigenspaces of 
\begin{displaymath}
A_V^*(g_n \dots g_1, y) A_V(g_n \dots g_1, y)
\end{displaymath}
which correspond to those eigenvalues which will converge as $n \to \infty$
to $2 \alpha_i$. 
Let $\cV_i(n,y,\bar{g}) = \cU_1(n,y,\bar{g}) \oplus \dots \oplus
\cU_i(n,y,\bar{g})$. Then, it follows from part II of
Theorem~\ref{theorem:osceledets:ae:walk} that for almost all $y$ and
almost all $\bar{g}$, 
\begin{equation}
\label{eq:subspaces:converge}
\lim_{n \to \infty} \cU_i(n,y,\bar{g}) = \cU_i(y,\bar{g}) \quad\text{
  and } \quad\lim_{n \to \infty} \cV_i(n,y,\bar{g}) = \cV_i(y,\bar{g}).
\end{equation}

\bold{The set $F_{good}(\epsilon,\sigma,L)$.}
Suppose $\epsilon > 0$, $\sigma > 0$, and $L \in \natls$. 
Let  $F_{good}(\epsilon,\sigma,L)$ 
denote the set of $y \in \cM$ such that for any
$v_y \in V$
\begin{equation}
\label{eq:FgoodL:one}
\mu^\natls\left(\{\bar{g} \st
d(v_y,\cV_{s-1}(L,y, \bar{g})) > \sigma) \} \right) > 1-\epsilon/2
\end{equation}
and also
\begin{equation}
\label{eq:FgoodL:two}
\mu^\natls\left(\{ \bar{g} \st \|A_V(g_L \dots g_1, y)\| \in
((\lambda_1-\epsilon/2)^L, (\lambda_1+\epsilon/2)^L)\}\right) > 1 - \epsilon/2.
\end{equation}

Since the cocycle $A_V$ is continuous and both (\ref{eq:FgoodL:one})
and (\ref{eq:FgoodL:two}) depend on $\bar{g}$ only via $g_1, \dots,
g_L$, the set $F_{good}(\epsilon,\sigma,L)$ is open. 

\begin{lemma}
\label{lemma:Fgood:L}
For any fixed $\epsilon > 0$ and $\delta > 0$ there exist $L_0 > 0$ and
$\sigma > 0$ such that for all $L > L_0$, 
$\nu_\cM(F_{good}(\epsilon,\sigma,L)) > 1-\delta$. 
\end{lemma}

\bold{Proof of Lemma~\ref{lemma:Fgood:L}.}
Let $\sigma > 0$ and $E_{good} \subset \cM$ be as in
Lemma~\ref{lemma:zero:one:law:star}, with $\delta/4$ and $\epsilon/4$
instead of $\delta$ and $\epsilon$.
By (\ref{eq:subspaces:converge}), we can find $L_1 > 0$ and a set $E_1
\subset \cM$ with
$\nu_\cM(E_1) > 1-\delta/4$ such that for $y \in E_1$ and $L \ge L_1$. 
\begin{displaymath}
\mu^\natls\left( \{\bar{g} \st
  d(\cV_{s-1}(L,y,\bar{g}),\cV_{s-1}(y,\bar{g})) > \sigma/2 \} \right) <
\epsilon/4.  
\end{displaymath}
Then, for $y \in E_{good} \cap E_1$, and $L \ge L_1$,
(\ref{eq:FgoodL:one}) holds (with $\sigma$ replaced by $\sigma/2$).   
Also, by Theorem~\ref{theorem:osceledets:ae:walk}, part I, there
exists $L_2 > 0$ and a subset $E_2 \subset \cM$ with $\nu_\cM(E_2)
> 1-\delta/2$ such that for $y \in E_2$ and $L \ge L_2$
(\ref{eq:FgoodL:two}) holds. Now let $F_{good}(\epsilon,\sigma,L) =
E_{good} \cap E_1 \cap E_2$ and choose $L_0 = \max(L_1,L_2)$.  
\qed\medskip

We also use the following trivial result:
\begin{lemma}
\label{lemma:bad:subspace}
For any $\sigma > 0$ there is a constant $c(\sigma) > 0$ with the following
property:  
Let $A \in GL(V)$ be a linear map, and let $\cV \subset V$ denote
the subspace spanned by the eigenspaces of all but the top eigenvalue
of $A^*A$. Then,  for any $v$ with
$\|v\|=1$ and $d(v,\cV) > \sigma$, we have 
\begin{displaymath}
\|A\| \ge \|A v\| > c (\sigma) \|A\|. 
\end{displaymath}
\end{lemma}

\bold{Proof of Lemma~\ref{lemma:Egood:L}.}
Suppose $\epsilon > 0$ and $\delta > 0$ are given, and let $\sigma >
0$ and $L_0 > 0$ be as in Lemma~\ref{lemma:Fgood:L}. 
Choose $L > L_0$ such that $(\lambda_1 - \epsilon/2)^L c(\sigma) >
(\lambda-\epsilon)^L$, where $c(\sigma)$ is as in
Lemma~\ref{lemma:bad:subspace}. Pick $v_y \in V$. Then, in view of
Lemma~\ref{lemma:bad:subspace}, for all $\bar{g}$
satisfying (\ref{eq:FgoodL:one}) and (\ref{eq:FgoodL:two}),
\begin{displaymath}
(\lambda_1+\epsilon/2)^L >
\|A_V(g_L \dots g_1,y) \| \ge \frac{\|A_V(g_L \dots g_1, y) v_y
  \|}{\|v_y\|} > (\lambda_1 - \epsilon)^L. 
\end{displaymath}
\qed\medskip

\subsection{Proof of Theorem~\ref{theorem:lyapunov:top:walk}}
Before proving Theorem~\ref{theorem:lyapunov:top:walk} we isolate the strong law of large numbers part of the argument. Pick an 
arbitrary $v_0\in V$ and let $v_i(\bar{g})=A(g_i\dots g_1,x)v_0$. Let $H(v_i)$ be as in the definition of $E_{good}(\epsilon,L)$.
\begin{lemma}\label{lemma:strong:law:os}Let $\mathcal{M}$ be given and $L$ be chosen so that $\nu_{\mathcal{M}}(E_{good}(\epsilon,L))>1-\epsilon$. 
For every $x\in \mathcal{M}$ almost every $\bar{g}$ we have that all but a set of $\mathbb{N}$ of density $4\epsilon$ is in disjoint 
blocks $[i+1,i+L]$ so that $g_{i+1}\dots g_{i+L} \in H(v_i)$ and $g_i\cdots g_1y \in E_{good}(\epsilon,L)$ .
\end{lemma}
\noindent
\textbf{Sublemma:} If $A \subset \mathbb{N}$ with density $c$ and $L \in \mathbb{N}$ then there exists $B \subset \mathbb{N}$  of density at least $1-c$ so that $B$ consist of disjoint blocks of length $L$ and so that the first term in each block is not in $A$.
 \begin{proof}Follow a greedy inductive algorithm. Given the ending point of the last $L$ block $n$, choose $m=\min\{k>n: k \notin A\}$ and add $[m,m+L-1]$.
  Let $B$ be the set given by this inductive procedure. By construction $(\mathbb{N}\setminus B)\subset A$ and the claim follows.
 \end{proof}
 \noindent
 \textbf{Sublemma:} For all $x$, almost every $\bar{g}$ we have that 
 $$\underset{n \to \infty}{\liminf}\, \frac{|\{i\leq n: g_i\cdots g_1x \in E_{good}(\epsilon,L) \text{ and }g_{i+1}\dots g_{i+L}\in H(v_i)\}|}
 {|\{i\leq n:g_i\cdots g_1x \in E_{good}(\epsilon,L)\}|}\geq 1-2\epsilon.$$
 \begin{proof}
 Lets enumerate $\{i:g_i\cdots g_1 x \in E_{good}(\epsilon,L)\}$ as $m_1\leq m_2\leq \dots$ By the definition of $E_{good}(\epsilon,L)$ for any $k_1,k_2,..,k_r$ we have
 $$\mu(\{\bar{g}: g_{m_{k_i}+1}\dots g_{m_{k_i}+L}\notin H(v_{m_{k_i}}) \text{ for }i\leq r\})<\epsilon^r.$$ 
 So for fixed $N$ we consider $\sum_{k=2\epsilon N}^N \epsilon^k {N\choose k}$ and the sublemma follows by Stirling's formula and the Borel-Cantelli Lemma.
 \end{proof}
 \begin{proof}[Proof of Lemma \ref{lemma:strong:law:os}]
 By the first Sublemma it suffices to show that the density $\{i: g_i\cdots g_1x \in E_{good}(\epsilon,L) \text{ and }g_{i+1}\dots g_{i+L}\in H(v_i)\}$
 is at least $1-4\epsilon$. 
 By the previous Sublemma it suffices to show that the density $\{i:g_i\cdots g_1x \in E_{good}(\epsilon,L)\}$ has density at least $1-2\epsilon$. 
 This follows from Corollary~\ref{cor:hits:often}.
 \end{proof}
\color{black}

Let $x \in \mathcal{M}$ be given, $\bar{g}$ be in the full measure set given by 
Lemma~\ref{lemma:strong:law:os} and $I(\bar{g})$ be a subset of natural numbers.

Now suppose $n \GG L$. Then, 
\begin{align*}
\log \|v_n \| = \sum_{i=1}^n \log \frac{\|v_i\|}{\|v_{i-1}\|} & =
\sum_{i \in I(\bar{g}) \cap [1,n-L]} \log  \frac{\|
  v_{i+L}\|}{\|v_i\|} & + &
\sum_{i \in K} \log \frac{\| v_{i}\|}{\|v_{i-1}\|} & + & \sum_{i=n-L}^L \log
\frac{\| v_{i}\|}{\|v_{i-1}\|} \\
& =  S_1 & + & S_2 & + & S_3.\\
\end{align*}
Let $C$ be such that for all $g$ in the support of $\mu$ and all $y
\in \cM$, $\|A(g,y)\| \le C$. Then, $|S_3| \le L \log C$. Also, since
the upper density of $K$ is at most $3\epsilon$, $|S_2| \le 3\epsilon
n \log C$. However, by (\ref{eq:def:Hv}),
\begin{displaymath}
S_1 \ge |I(\bar{g}) \cap [1,\dots, n]| \log (\lambda_1 - \epsilon)^L
\ge (1-4\epsilon)n (\lambda_1 - \epsilon). 
\end{displaymath}
Thus, for almost all $\bar{g}$ and any $n \GG L$, 
\begin{displaymath}
\frac{1}{n} \log\|A_V(g_n \dots g_1, x)\| \ge \frac{1}{n} \log \|v_n\|
\ge (1-4 \epsilon) (\lambda_1 - \epsilon) - 4 \epsilon \log C -
\frac{L}{n} \log C. 
\end{displaymath}
Since $\epsilon > 0$ is arbitrary, we get that for almost all $\bar{g}$,
\begin{displaymath}
\liminf_{n \to \infty} \frac{1}{n} \log \|A_V(g_n \dots g_1,x)\| \ge \lambda_1,
\end{displaymath}
which proves the lower bound in (\ref{eq:lyapunov:top:walk}). 
The proof of the upper bound in (\ref{eq:lyapunov:top:walk}) is similar. Let $a_0 = 1$, and $a_i =
\|A_V(g_i \dots g_1, x)\|$.  Then
\begin{align*}
\log a_n = \sum_{i=1}^n \log \frac{a_i}{a_{i-1}} & = \sum_{i \in
  I(\bar{g}) \cap [1,\dots,n-L]} \frac{a_{i+L}}{a_i} & + & \sum_{i \in
    K \cap [1,\dots n-L]}\log \frac{a_i}{a_{i-1}} & + & \sum_{i=n-L}^L
  \log \frac{a_i}{a_{i-1}} \\
& = S_1 & + & S_2 & + & S_3.
\end{align*}
As above, $|S_2| \le 4 \epsilon n \log C$ and $|S_3| \le L \log C$. By
(\ref{eq:def:Hv}), 
\begin{displaymath}
S_1 \le |I(\bar{g}) \cap [1,\dots, n]| \log (\lambda_1 + \epsilon)^L
\le n(\lambda_1+\epsilon). 
\end{displaymath}
Therfore, for almost all $\bar{g}$, 
\begin{displaymath}
\limsup_{n \to \infty} \frac{1}{n} \log a_n \le (\lambda_1 + \epsilon) +
4 \epsilon C. 
\end{displaymath}
Since $\epsilon > 0$ is arbitrary this completes the proof of
Theorem~\ref{theorem:lyapunov:top:flow}. 
\qed\medskip


\section{Proof of Theorem~\ref{theorem:birkhoff:flow}}

\subsection{An analogue of Lemma~\ref{lemma:limit:is:stationary}.}

Let $\eta_{T,\theta}$ denote the
measure on $SL(2,\reals)$ given by
\begin{displaymath}
\eta_{T,\theta}(\phi) = \frac{1}{T} \int_0^T \phi(g_t r_\theta)
\, dt.
\end{displaymath}
In this subsection we prove the following:
\begin{proposition}
\label{prop:ae:theta:P:invariant}
Fix $x \in \cM$. For almost every $\theta \in [0,2\pi]$, if
$\nu_\theta$ is any weak-star limit point (as $T \to \infty$) of
$\eta_{T,\theta} \ast \delta_x$, 
then $\nu_\theta$ is invariant under ${P}$, where
${P} = \begin{pmatrix} \ast & \ast \\ 0 & \ast \end{pmatrix} \subset
SL(2,\reals)$. 
\end{proposition}

The proof of Proposition~\ref{prop:ae:theta:P:invariant} is based on
the strong law of large
numbers. In fact, 
Proposition~\ref{prop:ae:theta:P:invariant}  holds for arbitrary
measure-preserving $SL(2,\reals)$ actions. 

It is clear from the definition, that for any $\theta$, any weak-*
limit point $\nu_\theta$ is invariant under $g_t$. 
Let 
\begin{displaymath}
u_\alpha = \begin{pmatrix} 1 & \alpha \\ 0 &
  1 \end{pmatrix} \qquad \bar{u}_\alpha = \begin{pmatrix} 1 & 0 \\ \alpha &
  1 \end{pmatrix}.
\end{displaymath}
Hence it is enough to show that $\nu_\theta$ is
invariant under $u_\alpha$ for every $\alpha$. Fix $0 < \alpha < 1$.

A simple calculation \mc{check} shows that for $0 < \xi < \pi/2$,
\begin{displaymath}
r_\xi = \bar{u}_{-\tan \xi} a_\xi {u}_{\tan \xi}, \text{ where } a_\xi
= \begin{pmatrix} \cos \xi & 0 \\ 0 & 1/\cos \xi \end{pmatrix}.
\end{displaymath}
Then,
\begin{equation}
\label{eq:gt:rxi}
g_t r_\xi = (g_t \bar{u}_{-\tan \xi} g_t^{-1}) a_\xi (g_t {u}_{\tan \xi}
g_t^{-1} ) g_t = (\bar{u}_{-e^{-2t} \tan \xi}) a_\xi ({u}_{e^{2t} \tan \xi}) g_t.
\end{equation}
Let $\alpha_t$ be defined by the equation
\begin{equation}
\label{eq:def:theta:t}
e^{2t} \tan \alpha_t = \alpha.
\end{equation}
We claim that Proposition~\ref{prop:ae:theta:P:invariant} follows
quickly from the following: 
\begin{proposition}
\label{prop:shifted:averages}
Fix $x \in \cM$, and $0 < \alpha < 1$. 
Let $\phi \in C_c(\cM)$ be a test function. Let 
\begin{equation}
\label{eq:def:ft}
f_t(\theta) = \phi(g_t r_\theta x) - \phi(g_t r_{\theta+\alpha_t} x)
\end{equation}
where $\alpha_t$ is as in (\ref{eq:def:theta:t}). 
Then, for almost every $\theta \in [0,2\pi]$,
\begin{equation}
\label{eq:prop:shifted:averages}
\lim_{T \to \infty} \frac{1}{T} \int_0^T f_t(\theta) \, dt = 0. 
\end{equation}
\end{proposition}

\bold{Proof that Proposition~\ref{prop:ae:theta:P:invariant} follows
  from Proposition~\ref{prop:shifted:averages}.}
Let $x$, $\phi$, $\alpha$, $\alpha_t$ be as in
Proposition~\ref{prop:shifted:averages}. We need to prove that for
almost all $\theta$, 
\begin{equation}
\label{eq:bar:ualpha:need:to:prove}
\lim_{T \to \infty} \frac{1}{T} \int_{0}^T (\phi({u}_\alpha
g_t r_\theta x) - \phi(g_t r_\theta x)) \, dt = 0.
\end{equation}
Since the smooth functions are dense in $C_c(\mathcal{M})$, 
without loss of generality, we may assume that $\phi$ is smooth. Then,
there exists a constant $M$ such that for $h \in SL(2,\reals)$ near
the identity $I \in SL(2,\reals)$ and all $y \in \cM$,
\begin{equation}
\label{eq:phi:Lipshitz}
|\phi(h y) - \phi(y)| \le M \|h-I\|.
\end{equation}
We write
\begin{equation}
\label{eq:barualpha:two:terms}
\phi({u}_\alpha g_t r_\theta x) - \phi(g_t r_\theta x) = 
(\phi(u_\alpha g_t r_\theta x) - \phi(g_t r_{\theta +
  \alpha_t} x)) + (\phi(g_t r_{\theta+\alpha_t} x) - \phi(g_t r_\theta
x)). 
\end{equation}
Let $J_1$ be the contribution of the first term in parenthesis in
(\ref{eq:barualpha:two:terms}) 
to (\ref{eq:bar:ualpha:need:to:prove}) and let $J_2$ be the
contribution of the second term. We have, using (\ref{eq:gt:rxi}) and
(\ref{eq:def:theta:t}),
\begin{align*}
J_1 & = \lim_{T \to \infty} \frac{1}{T} \int_0^T \phi(u_{\alpha}
g_t r_\theta x) - \phi((\bar{u}_{e^{-2t} \tan \alpha_t}) a_{\alpha_t}
u_\alpha g_t r_\theta x) \, dt \\
& \le M \lim_{T \to \infty}
\frac{1}{T} \int_0^T \| \bar{u}_{e^{-2t} \tan \alpha_t}
  a_{\alpha_t} - I \| \, dt = 0,\\
\end{align*}
by (\ref{eq:phi:Lipshitz}) and $\alpha_t = O(e^{-t})$. Also $J_2 = 0$ by
Proposition~\ref{prop:shifted:averages}. Thus
(\ref{eq:bar:ualpha:need:to:prove}) holds.

This shows that for any fixed $0< \alpha < 1$ for almost all $\theta$,
the measures $\nu_\theta$ of 
Proposition~\ref{prop:ae:theta:P:invariant} are invariant under
${u}_\alpha$ (as well as $g_t$ for all $t$). We now repeat the proof with
two different $\alpha$'s linearly independent over $\ratls$. We get
that for almost all $\theta$, any limit point of $\eta_{T,\theta} \ast
\delta_x$ is invariant under a dense subgroup of ${P}$, hence
invariant under all of ${P}$. This
completes the proof of Proposition~\ref{prop:ae:theta:P:invariant}.
\qed\medskip

Note that from (\ref{eq:def:ft}) we have
\begin{equation}
\label{eq:ft:zero:mean}
\int_0^{2\pi} f_t(\theta) \, d\theta = 0.
\end{equation}

\begin{lemma}
\label{lemma:ft:fs:correlation:bound}
There exist $\lambda > 0$ and $C > 0$ such that
\begin{equation}
\label{eq:ft:fs:correlation:bound}
\int_{0}^{2\pi} f_t(\theta) f_s(\theta) \, d\theta \le C e^{-\lambda |s-t|}.
\end{equation}
\end{lemma}
Figure 1 below helps describe the proof.
\color{black}

\bold{Proof.} Without loss of generality we may assume that
(\ref{eq:phi:Lipshitz}) holds, and that $t > s$. Let $r =
(t+s)/2$. Let $A_\varphi \subset [0,2\pi]$ be an interval of the form 
$[\varphi - e^{-2r}, \varphi + e^{-2r}]$. Then, for $\theta =
\varphi+\eta \in A_\varphi$,
\begin{displaymath}
g_s r_\theta = g_s r_\eta r_\varphi = (g_s r_\eta g_s^{-1}) g_s r_\varphi,
\end{displaymath}
and hence, using (\ref{eq:phi:Lipshitz}),
\begin{displaymath}
|f_s( \theta) - f_s(\varphi) | \le 4 M \|g_s r_\eta g_s^{-1} - I \|
\le 4 M e^{-2(r-s)}.
\end{displaymath}
Therefore, 
\begin{equation}
\label{eq:tmp:step:int:A:varphi}
\frac{1}{|A_\varphi|} \int_{A_\varphi} f_t(\theta) f_s(\theta) \,
d\theta = f_s(\varphi) \frac{1}{|A_\varphi|} \int_{A_\varphi}
f_t(\theta) \, d\theta + O(e^{-2(r-s)}).
\end{equation}
Now, from the definition (\ref{eq:def:ft}) of $f_t$, we have
\begin{equation}
\label{eq:telescoping}
\frac{1}{|A_\varphi|}\int_{A_\varphi} f_t(\theta) \, d\theta = O(e^{-2(t-r)}).
\end{equation}
(Essentially the integral cancels except for the contribution of two 
``boundary regions'' each of size $O(\theta_t) = O(e^{-2t})$. Since $f_t$ is
bounded and $|A_\varphi| = 2 e^{-2r}$, (\ref{eq:telescoping}) follows.)
Now from (\ref{eq:tmp:step:int:A:varphi}) and (\ref{eq:telescoping})
we get that for every $\varphi \in [0,2\pi]$,
\begin{displaymath}
\frac{1}{|A_\varphi|} \int_{A_\varphi} f_t(\theta) f_s(\theta) \,
d\theta = O(e^{-2(t-r)}) + O(e^{-2(r-s)}).
\end{displaymath}
Since $r = (s+t)/2$, this  immediately implies
(\ref{eq:ft:fs:correlation:bound}). 
\qed
\medskip

\begin{figure}[h]\label{fig:circles}\caption{We integrate over an interval $A_{\varphi} \subset S^1$ so that $g_rA_{\varphi}$ is of size 1. By our Lipshitz assumption on $\phi$, $f_s$ is basically constant on $A_{\varphi}$. Since $\alpha_t$ is negligible compared to $|A_{\varphi}|$ the integral in Equation (\ref{eq:bar:ualpha:need:to:prove}) cancels except for two negligible boundary terms.}
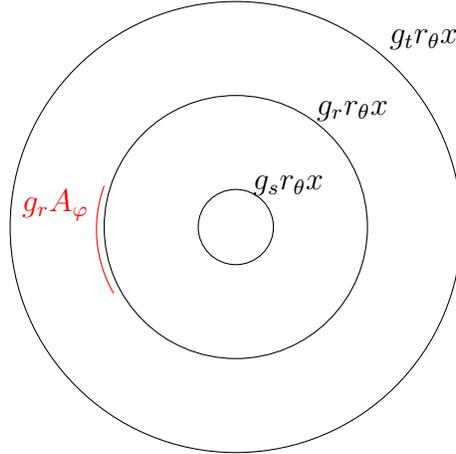
\begin{tikzpicture}
\draw (0,0) circle  (.5 cm)
circle (1.75 cm)
circle (3cm); 
\path (.70,.55) node {$g_sr_{\theta}x$};
\path (1.55,1.55) node
{$g_r r_{\theta}x$};
\path (2.5,2.5) node {$g_t r_{\theta}x$};
\draw[red,-] (-1.75,.55) arc (160:210:1.7);
\path[-,red] (-2.4,.3) node {$g_rA_\varphi$};
\end{tikzpicture}
\end{figure}
\color{black}

Now, Proposition~\ref{prop:shifted:averages} follows from the 
following straightforward version of the strong law of large numbers,
which we will prove in \S\ref{sec:subsec:strong:law:exponential:decay}
for the interested readers' convenience:

\begin{lemma}
\label{lemma:strong:law:exponential:decay}
Suppose $f_t: [0,2\pi] \to \reals$ are bounded functions satisfying
(\ref{eq:ft:zero:mean}) and (\ref{eq:ft:fs:correlation:bound}) (for
some $C > 0$ and $\lambda > 0$). Additionally, assume that $f_t(\theta)$ are $2M$-Lipshitz functions of $t$ for each $\theta$ (\ref{eq:phi:Lipshitz}).
 Then, for almost every $\theta \in [0,2\pi]$, (\ref{eq:prop:shifted:averages}) holds.
\end{lemma}

\subsection{Proof of Lemma~\ref{lemma:strong:law:exponential:decay}}
\label{sec:subsec:strong:law:exponential:decay}
We recall the following basic facts:
\begin{lemma}[Chebyshev inequality] Let $f:\Omega\to \mathbb{R}$ have 
$\int_{\Omega}f(\omega)^2d\nu\leq C$. Then 
$$\nu(\{\omega:|f(\omega)|>sC\})\leq \frac {1}{s^2C}.$$
\end{lemma}
\begin{lemma}[Borel-Cantelli] Let $A_1,...$ be $\mu$-measurable sets
  such that  
${\sum_{i=1}^{\infty}\mu(A_i)<\infty.}$ Then 
$\mu(\cap_{i=1}^{\infty}\cup_{n=i}^{\infty}A_n)=0$. 
\end{lemma} 

\begin{proof}[Proof of Lemma~\ref{lemma:strong:law:exponential:decay}]
First, because $f_t(\theta)$ is an $2M$-Lipshitz function of $t$ for
each $\theta$ it suffices to show that for any $\epsilon$ and almost
every $\theta$ we have:
\begin{equation}
\label{eq:strong:law:need:to:prove}
\underset{n \to \infty}{\limsup}\left|\frac 1 n \sum_{i=1}^n f_{\epsilon i}(\theta)\right|<\epsilon. 
\end{equation}
We will show that (\ref{eq:strong:law:need:to:prove}) follows from 
(\ref{eq:ft:fs:correlation:bound}), 
the Borel-Cantelli lemma, and Chebyshev's inequality.

To see this, observe that $\int (\sum_{i=1}^n f_{\epsilon i}(\theta))^2= \int \sum_{i=1}^nf_{\epsilon i}(\theta)^2+2C\sum_{i<j<n}e^{-\lambda |j\epsilon-i\epsilon|}.$
So, there exists $C'_{\epsilon}$ such that $\int (\sum_{i=1}^n f_{\epsilon i}(\theta))^2\leq nC'_{\epsilon}$. By Chebyshev's inequality:

\noindent there exists $C''_{\epsilon}$ such that
$\lambda(\{\theta: \left|\frac 1 n\sum_{i=1}^n f_{\epsilon
    i}(\theta)\right|>\frac{\epsilon}{2}\})\leq
\frac{C''_{\epsilon}}{n}$. 
By the Borel-Cantelli Lemma
\noindent it follows that for almost every $\theta$ we have 
$$ \underset{n \to \infty}{\limsup}\left|\frac 1 {n^2}
  \sum_{i=1}^{n^2} f_{\epsilon i}(\theta)\right|\leq
\frac{\epsilon}{2}.$$ 
Notice $(N+1)^2-N^2=2N+1$ and so for any $M>0$ there exists $N\in \mathbb{N}$ such that $0\leq M-N^2<2\sqrt{M}$.
 It follows that for all large enough $M$ 
$$\left|\frac 1 M\sum_{i=1}^M f_{i\epsilon}(\theta)\right| \leq \left|\frac
1 {M} \sum_{i=1}^{N^2}f_{i\epsilon}(\theta)+\frac 1 M
\sum_{i=N^2+1}^M f_{i\epsilon}(\theta)\right| \leq \frac{\epsilon}{2}
+\frac{2\hat{C}\sqrt{M}}{M}.$$ 
This uses that the $f_{t}$ are uniformly bounded.
For all large enough $M$ this is smaller than $\epsilon$ and
Lemma~\ref{lemma:strong:law:exponential:decay} follows.
\end{proof}

\subsection{Completion of the proof of
  Theorem~\ref{theorem:birkhoff:flow}.}

\begin{prop}[\protect{\cite[Proposition 2.13]{EMM}}]
\label{prop:EMiMo:main} 
Let $\cN$ be any affine submanifold. 
Then there exists an $SO(2)$ invariant function
$f_\cN: \cH_1(\alpha) \to \reals_+$, $c,b \in \mathbb{R}$  such that
\begin{enumerate}
\item $f_\cN(x)=\infty$ iff $x \in \cN$. Also $f_\cN$ is bounded on
  compact subsets of $\mathcal{H}_1(\alpha) \setminus \cN$. Also
  $\overline{\{x:f_\cN(x)\leq N\}}$ is compact for any $N$. 
\item There exists  $b > 0$ (depending on $\cN$) and for
  every $0 < c < 1$   there exists $t_0 > 0$ (depending on $\cN$ and
  $c$) such that for all $x \in \cH_1(\alpha)$ and all $t > t_0$,
$\frac{1}{2\pi} \int_0^{2\pi} f_\cN(g_t r_\theta x)\, d\theta \leq c
f_\cN(x)+b$,  
\item For any $g \in SL(2,\mathbb{R})$ and $\|g\|\leq 1$ we have
  $f_\cN(gx)\leq \sigma' f_\cN(x)$. 
\end{enumerate}
\end{prop}

\begin{theorem}[\protect{\cite[Theorem 2.3]{Ath thesis}}] 
\label{theorem:athreya}
Given a function $f_\cN$ satisfying (2) and (3) of
Proposition~\ref{prop:EMiMo:main}, 
 we have that for any $0<\beta<1$ there exist $M < \infty$ and $\gamma
 < 1$ such that for every $x$ we have 
$$\lambda\left(\left\{\theta\st f(g_t
  r_{\theta}x)>M \quad \text{\rm for at least $\beta$-fraction of $t \in [0,T]$} \right\}\right)<\gamma^T $$ 
for all large enough $T$.
\end{theorem}
The following is similar of the proof of Theorem~\ref{theorem:birkhoff:walk}.

\bold{Proof of Theorem~\ref{theorem:birkhoff:flow}.}
Let $\nu_\theta$ be any weak-star limit point of the measures 
$\eta_{T,\theta} \ast \delta_x$. By
Proposition~\ref{prop:ae:theta:P:invariant}, for almost all $\theta$,
$\nu_\theta$ is ${P}$-invariant. 

By \cite[Theorem~1.4]{EM}, any ergodic ${P}$-invariant measure is  
$SL(2,\reals)$-invariant and affine. 
Therefore, since $\nu_\theta$ is supported on
$\cM$, it has can be decomposed into ergodic components as
\begin{displaymath}
\nu_\theta = \sum_{\cN \subseteq \cM} a_\cN(\theta) \, \nu_\cN,
\end{displaymath}
where $a_\cN(\theta) \in [0,1]$ and 
the sum is over the affine invariant submanifolds $\cN$ contained in
$\cM$. (Here $\cN = \cM$ is allowed). By \cite[Proposition~2.16]{EMM}
this is a countable sum. By applying Theorem~\ref{theorem:athreya}
for the case $\cN = \emptyset$ we get that for almost all $\theta$, 
$\nu_\theta$ is a probability measure. 
Indeed, for every $\beta>0$ if $L$ is large enough then 
$$\lambda(\{\theta:  
|\{t<L:f(g_sr_{\theta}x)>M\}|>\beta L\})\leq \gamma^L.$$ Thus if $S$ is large enough 
$$\lambda(\{\theta: \exists L>S \text{ such that } 
|\{t<L:f(g_sr_{\theta}x)>M\}|>\beta L\})\leq \sum_{j=S}^{\infty}\gamma_{M}^j<\infty$$ 
and the claim follows by Borel-Cantelli.
\color{black}
Then, by applying
Theorem~\ref{theorem:athreya} again with $\cN$  any affine invariant
submanifold properly contained in $\cM$,  we see that for
almost all $\theta$, $\nu_\theta(\cN) = 0$. 
Thus, for almost all $\theta$,  $a_\cN(\theta) = 0$ for any $\cN$ 
properly contained in $\cM$. Since $\nu_\theta$ is a probability
measure, this forces $\nu_\theta = \nu_\cM$ for almost all $\theta$, 
completing the proof of Theorem~\ref{theorem:birkhoff:flow}.  
\qed\medskip

\section{Proof of Theorem~\ref{theorem:lyapunov:top:flow}}
\label{sec:proof:osceledts:flow}

Let $\mu$ be as in \S\ref{sec:random:walks}. The following lemma
expresses the well known fact that a typical random walk trajectory
tracks a geodesic (up to sublinear error).  
\begin{lemma}[Sublinear Tracking]
\label{lemma:sublinear:tracking}
There exists $\lambda >0$ (depending only on $\mu$), and 
$\mu^\natls$-almost all $\bar{g} = (g_1, \dots, g_n, \dots) \in SL(2,\reals)^\natls$ there
exists $\bar{\theta} = \bar{\theta}(\bar{g}) \in [0,2\pi)$ such that 
\begin{equation}
\label{eq:lemma:sublinear:tracking}
\lim_{n \to \infty} \frac{1}{n} \log \| (g_{\lambda n} r_{\bar{\theta}})
(g_n \dots g_1)^{-1} \| = 0,
\end{equation}
where $g_{\lambda n}$ denotes $\begin{pmatrix} e^{\lambda
    n} & 0 \\ 0   & e^{-\lambda n} \end{pmatrix}$. Furthermore, the
distribution of $\bar{\theta}$ is uniform, i.e
\begin{equation}
\label{eq:uniform:hitting:measure}
\mu^\natls\left( \{ \bar{g} \in SL(2,\reals)^\natls \st
  \bar{\theta}(\bar{g}) \in [\theta_1, \theta_2] \} \right) =
\frac{|\theta_2 - \theta_1|}{2\pi}.
\end{equation}
\end{lemma}

\bold{Proof.} We apply the multiplicative ergodic theorem
Theorem~\ref{theorem:osceledets:ae:walk} to the identity cocycle
$\alpha(g,x) = g$ (instead of $A_V$). Let $\Lambda \in SL(2,\reals)$
be as in II of Theorem~\ref{theorem:osceledets:ae:walk}. Since
$\Lambda$ is symmetric, we may write
\begin{displaymath}
\Lambda(\bar{g}) = r_{\bar{\theta}}^{-1} \begin{pmatrix} e^{\lambda} & 0 \\
  0 & e^{-\lambda} \end{pmatrix} r_{\bar{\theta}}.
\end{displaymath}
Then, (\ref{eq:general:sublinear:tracking}) immediately implies
(\ref{eq:lemma:sublinear:tracking}). 

Let $\sigma$ denote the measure on $[0,2\pi)$ such that
$\sigma([\theta_1,\theta_2])$ is given by the left-hand-side of
(\ref{eq:uniform:hitting:measure}). It is easy to show that $\sigma$
must be $\mu$-stationary, i.e. $\mu \ast \sigma = \sigma$. Since $\mu$
is assumed to be $SO(2)$-bi-invariant, this implies that $\sigma$ is
the uniform measure.
\qed\medskip

\bold{Proof of Theorem~\ref{theorem:lyapunov:top:flow}.} By
Theorem~\ref{theorem:lyapunov:top:walk} there exists a set $E$ with
$\mu^\natls(E) =0$ such that for $\bar{g} \not\in E$,
(\ref{eq:lyapunov:top:walk}) holds. 
By Lemma~\ref{lemma:sublinear:tracking}, for almost all $\theta \in
[0,2\pi)$  there exists $\bar{g} = (g_1, \dots, g_n, \dots) \not\in E$
so that if we write 
\begin{displaymath}
g_{\lambda n} r_\theta = \epsilon_n \, g_n \dots g_1,
\end{displaymath}
then $\epsilon_n \in SL(2,\reals)$ satisfies
\begin{equation}
\label{eq:epsilon:n:small}
\lim_{n \to \infty} \tfrac{1}{n} \log \| \epsilon_n \| = 0. 
\end{equation}
Then, by the cocycle relation,
\begin{displaymath}
A_V(g_{\lambda n}, r_\theta x) = A_V(\epsilon_n, g_n \dots g_1 x) A_V(
g_n \dots g_1, x). 
\end{displaymath}
There exists $C > 0$ and $N < \infty$ 
so that for all $g \in SL(2,\reals)$ and all $x \in
\cH_1(\alpha)$, we have
\begin{equation}
\label{eq:cocycle:upper:bound}
||A_V(g,x)|| \le C \| g\|^N 
\end{equation}
Hence, by (\ref{eq:epsilon:n:small}) and
(\ref{eq:cocycle:upper:bound}), we have
\begin{displaymath}
\log \| A_V(g_{\lambda n}, r_\theta x) \| = \log \| A_V(g_n \dots
g_1,x) \| + o(n).
\end{displaymath}
Now the existence of the limit in (\ref{eq:theorem:lyapunov:top:flow})
follows from (\ref{eq:lyapunov:top:walk}). 
\qed
\end{document}